\documentclass[reqno,12pt]{article} 
\usepackage{amsfonts}
\usepackage{amsthm}
\usepackage{amsmath}
\usepackage{amssymb}
\usepackage{enumitem}
\usepackage[]{authblk}
\usepackage{bbm}
\usepackage{pgf,tikz}
\usetikzlibrary{arrows}
\usetikzlibrary{shapes.misc}

\tikzset{cross/.style={cross out, draw=black, minimum size=2*(#1-\pgflinewidth), inner sep=0pt, outer sep=0pt},
cross/.default={5pt}}

\usepackage[numeric,initials,nobysame]{amsrefs}
\usepackage{hyperref}
\usepackage[capitalise]{cleveref}

\newtheorem{thm}{Theorem}
\crefname{thm}{Theorem}{Theorems}
\newtheorem{lem}{Lemma}
\crefname{lem}{Lemma}{Lemmas}
\newtheorem{cor}[lem]{Corollary}
\crefname{cor}{Corollary}{Corollaries}

\crefname{prop}{Proposition}{Propositions}
\newtheorem{conj}[lem]{Conjecture}
\crefname{conj}{Conjecture}{Conjectures}
\newtheorem{ques}[lem]{Question}
\crefname{ques}{Question}{Questions}

\crefname{fact}{Fact}{Facts}
\newtheorem{algo}{Algorithm}
\crefname{algo}{Algorithm}{Algorithms}
\theoremstyle{definition}

\crefname{defn}{Definition}{Definitions}
\newtheorem{rem}[lem]{Remark}
\crefname{rem}{Remark}{Remarks}

\crefname{ex}{Example}{Examples}

\crefname{obs}{Observation}{Observations}
\newtheorem{claim}[lem]{Claim}
\crefname{claim}{Claim}{Claims}

\numberwithin{lem}{section}

\newcommand{\cF}{\ensuremath{\mathcal F}}

\newcommand{\bbE}{{\ensuremath{\mathbb E}} }

\newcommand{\bbL}{{\ensuremath{\mathbb L}} }

\newcommand{\bbN}{{\ensuremath{\mathbb N}} }

\newcommand{\bbP}{{\ensuremath{\mathbb P}} }

\newcommand{\bbT}{{\ensuremath{\mathbb T}} }

\newcommand{\bbV}{{\ensuremath{\mathbb V}} }

\newcommand{\bbZ}{{\ensuremath{\mathbb Z}} }

    \let\d=\delta  
      \let\k=\kappa  \let\l=\lambda
      \let\o=\omega      
  \let\s=\sigma    
\let\y=\upsilon \let\x=\xi 
 \let\F=\Phi  \let\G=\Gamma

\newcommand{\tc}{t_{\mathrm{c}}}
\newcommand{\pc}{p_{\mathrm{c}}}
\renewcommand{\sc}{s_{\mathrm{c}}}

\newcommand{\Eo}{E^{\mathrm{out}}}
\newcommand{\Ei}{E^{\mathrm{in}}}
\newcommand{\1}{{\ensuremath{\mathbbm{1}}} }

\renewcommand{\leq}{\leqslant}
\renewcommand{\geq}{\geqslant}
\renewcommand{\le}{\leqslant}
\renewcommand{\ge}{\geqslant}
\renewcommand{\to}{\rightarrow}

\begin{document}
\title{Weakly constrained-degree percolation on the hypercubic lattice}

\author[,1]{Ivailo Hartarsky\thanks{\textsf{hartarsky@ceremade.dauphine.fr}}}
\author[,2]{Bernardo N. B. de Lima\thanks{\textsf{bnblima@mat.ufmg.br}}}
\affil[1]{CEREMADE UMR 7534, Universit\'e Paris-Dauphine, CNRS, PSL University\protect\\
Place du Mar\'echal de Lattre de Tassigny, 75016 Paris, France}
\affil[2]{Universidade Federal de Minas Gerais, Departamento de Matem\'atica\\
Av. Ant\^onio Carlos 6627, Belo Horizonte-MG, Brazil}
\date{\vspace{-0.25cm}\today}
\maketitle
\vspace{-0.75cm}
\begin{abstract}
We consider the Constrained-degree percolation model on the hypercubic lattice, $\mathbb L^d=(\mathbb Z^d,\mathbb E^d)$ for $d\geq 3$. It is a continuous time percolation model defined by a sequence, $(U_e)_{e\in\mathbb E^d}$, of \emph{i.i.d.} uniform random variables in $[0,1]$ and a positive integer (constraint) $\kappa$. Each bond $e\in\mathbb E^d$ tries to open at time $U_e$; it succeeds if and only if both its end-vertices belong to at most $\kappa -1$ open bonds at that time.

Our main results are quantitative upper bounds on the critical time, characterising a phase transition for all $d\geq 3$ and most nontrivial values of $\kappa$. As a byproduct, we obtain that for large constraints and dimensions the critical time is asymptotically $1/(2d)$. For most cases considered it was previously not even established that the phase transition is nontrivial.

One of the ingredients of our proof is an improved upper bound for the critical curve, $s_{\mathrm{c}}(b)$, of the Bernoulli mixed site-bond percolation in two dimensions, which may be of independent interest.

\end{abstract}


\section{Introduction}
\label{sec:intro}
\subsection{Model}
\label{subsec:model}
The constrained-degree percolation model was introduced in \cite{Teodoro14} as follows. Consider an infinite transitive connected graph $G=(\bbV,\bbE)$, let $\kappa$ be a positive integer such that $\kappa\leq \deg(G)$, where $\deg(v)=|\{u\in\bbV:uv\in\bbE\}|$ and $\deg (G)= \deg (v)$ for all $v\in\bbV$, since $G$ is transitive.

Let $(U_e)_{e\in\bbE}$ be a sequence of independent and identically distributed uniform random variables on $[0,1]$. For each $t\in [0,1]$, define a continuous time percolation model, denoting by $\omega^{G,\kappa}(t)\in\{0,1\}^{\bbE}$ the configuration of {\em open} (1) or {\em closed} (0) bonds.

At time $t=0$, we declare all bonds as closed (i.e.\ $\omega_{e}^{G,\kappa}(0)=0$ for all $e\in\bbE$). As time progresses, bonds will become open. Each bond $e\in\bbE$ will try to open at time $U_e$, it will succeed if and only if both its end-vertices have degree, in the cluster of open bonds, at most $\kappa -1$.

More formally, the model is described by the probability space $(\Omega, \cF,\bbP)$, where $\Omega=[0,1]^{\bbE}$ is the space of clocks, $\cF$ is the $\sigma$-algebra generated by cylinder sets of $\Omega$ and $\bbP$ is the product of Lebesgue measures on $[0,1]$. Given the sequence of clocks $(U_e)_{e\in\bbE}$, the percolation configuration on the bond $v_1v_2$ at time $t\in [0,1]$, denoted $\o_{v_1v_2}^{G,\k}(t)$, is the indicator function of the intersection of the events
\[\{U_{v_1v_2}\leq t\}\] and \[\left\{\left|\left\{u\in\bbV\setminus\{v_{3-i}\} :\omega_{v_iu}^{G,\kappa}(U_{v_1v_2})=1\right\}\right|<\k\right\}\mbox{ for } i\in\{1,2\}.\]

Using  the Harris graphical construction, one can establish that this model is well defined (see e.g.\ \cite{Liggett05}). On the other hand, it has a dependence of infinite range and does not satisfy the FKG inequality, nor the insertion tolerance (or finite energy) property. When $\kappa\ge\deg(G)$ the constrained-degree percolation model at time $t$ reduces to the ordinary Bernoulli bond percolation model with parameter $t$.

Given $\o\in\{0,1\}^\bbE$, the notation $0\leftrightarrow\infty$ means that there are infinitely many vertices connected to origin by paths of open edges in $\o$. We simplify the notation denoting the event $\{0\leftrightarrow\infty\mbox{ in }\omega^{G,\kappa}(t)\}$ by $\{0\leftrightarrow\infty\mbox{ at }t\}$. 

The {\em probability of percolation} is the function $\theta^{G,\kappa}(t):[0,1]\rightarrow[0,1]$, where $\theta^{G,\kappa}(t)=\bbP(0\leftrightarrow\infty\mbox{ at }t)$. By definition, the function $\theta^{G,\kappa}(t)$ is non-decreasing in $t$, then it is natural to define the {\em critical time} \[\tc^{G,\k}:=\sup\{t\in [0,1]:\theta^{G,\kappa}(t)=0\},\] with the convention $\sup \varnothing = +\infty$. Whenever they are clear from the context, we will drop the indices $G$ and $\kappa$ from the notation.

Throughout this work, we will almost exclusively deal with the hypercubic lattice $\bbL^d=(\bbZ^d,\bbE^d)$, where $\bbE^d=\{uv\in\bbZ^d\times\bbZ^d: \|u-v\|_1=1\}$.

\subsection{Related Works}
\label{subsec:background}
In \cite{Teodoro14}, it was proved that for the hypercubic lattice $\bbL^d,\ d\geq 2$, when $\kappa=2d-1$, there is percolation at time $t=1$, that is, $\theta^{\bbL^d, 2d-1}(1)>0$.  In \cite{DeLima20}, it was shown that there is a nontrivial phase transition on the square lattice $\bbL^2$ in the nontrivial case $\kappa=3$. More precisely (see Theorem 1 therein), it was proved that $t_{\mathrm{c}}^{\bbL^2,3}\in \left(\frac{1}{2},1\right)$. With a martingale argument, it was also proved in \cite{DeLima20} that for all dimensions $d\geq 2$ and $\kappa=2$, $t_{\mathrm{c}}^{\bbL^d,2}=+\infty$, that is, $\theta^{\bbL^d,2}(t)=0$ for all $t\in [0,1].$ We emphasise that nothing it is known about percolation for other values of $\kappa$ or $t<1$ when $d\geq 3$ prior to the present work.

In \cite{DeLima20}, the uniqueness of the infinite cluster was also studied, as well as the constrained-degree percolation on the regular $d$-ary trees, $\bbT^d$, for which it is proved that $\tc^{\bbT^d,3}<1$ for all $d\geq 2$.

The idea of random system of constrains has ancient origins in the Physical literature and goes back to the work of Flory \cite{Flory39} in which the dimer (or domino) tiling problem was introduced. In 1979, the paper \cite{Gaunt79} introduced the Percolation with restricted-valence model, that is essentially the same model as the constrained-degree percolation studied here, but it is a site percolation version instead of bond percolation.

Some recent mathematical works on variations on percolative models with some kind of constrains on the vertices are \cites{Grimmett10, Grimmett17,Garet18a,Holroyd21}.

\subsection{Results}
\label{subsec:results}
The main goal of this work is to prove that there is a phase transition ($\tc^{\bbL^d,\kappa}<1$) for the hypercubic lattice, $\bbL^d$, for $d\geq 3$ and some nontrivial values of $\kappa$, that is $3\leq\kappa\leq 2d-1$. Moreover, we seek non-perturbative results applying beyond $t\approx 1$, as well as for $\kappa$ much smaller than the least constraint case, $\k=2d-1$. Indeed, we even manage to treat constraints not diverging with $d\to\infty$. From now on, we will denote the critical time for the hypercubic lattice $\bbL^d$, $\tc^{\bbL^d,\k}$, by $\tc^\kappa (d)$.

\begin{thm}
\label{th:general}
Let $c=1.7$ and $\k\ge 10$. Then for any $d>\k/2$ we have 
\[\tc^{\k}(d)\le c/d.\]
Moreover, still with $c=1.7$, for lower dimensions we have the stronger results
\[\tc^{\k}(d)\le\frac{c}{d}\quad \text{ for }\begin{cases}
d=4&\k=7,\\
d\in\{5,6\}&\k\ge 8,\\
d\in\{7,8,9,10,11,12,14,16\}&\k\ge 9.
\end{cases}
\]
\end{thm}
Notice that by a standard branching process comparison for ordinary percolation, it is easy to show that $\tc^{\k}(d)\ge 1/(2d-1)$ for all $\k$ and $d$, so that \cref{th:general} shows that $\tc^{\k}(d)=\Theta(1/d)$ as $d\to\infty$. In fact, for high dimensions and weak constraints, the following sharp result is obtained as a byproduct of the proof of \cref{th:general}.
\begin{thm}
\label{th:high}
For any two integer sequences $(\k_n)$ and $(d_n)$ such that $\k_n\to\infty$ and $d_n\to\infty$ as $n\to\infty$, we have
\[\lim_{n\to\infty}d_n\cdot t_{\mathrm{c}}^{\k_n}(d_n)=\frac{1}{2}.\]
\end{thm}

In order to illustrate the fact that our approach is not intrinsically high-dimensional, we further adapt it to obtain a nontrivial result even in three dimensions.
\begin{thm}
\label{th:3d}
Let $d=3$ and $\k=5$. Then
\[\tc^{\k}(d)\le 0.62.\]
Moreover, the same holds for the graph  $\bbL^2_{\boxtimes} =(\bbZ^2,\{uv\in\bbZ^2\times\bbZ^2:\|u-v\|_\infty =1\})$, the matching graph of $\bbL^2$, obtained from the square lattice by adding the diagonals of each face, and $\k=7$.
\end{thm}

Finally, let us mention that in \cref{app} we establish a quantitative improvement of a result of Chayes and Schonmann \cite{Chayes00} on mixed site-bond percolation in the case of $\bbL^2$, which may be of independent interest. It is used in the proof of \cref{th:general} to allow the treatment of smaller values of $\k$.

\subsection{Ideas of the proofs}
\label{subsec:sketch}
Let us give an overview of the proofs of our main results. In both cases the idea, though peculiar, is quite simple. In this section we prefer to omit some technical issues in order not to obscure the essence, hoping that this will not lead to confusion.

\subsubsection{General result}
\label{subsubsec:sketch:general}
We first sketch the proof of \cref{th:general}. Although the same proof will directly apply to all sets of parameters in the statement of the theorem, the reader is advised to think of $t=c/d$ with $c$ sufficiently large, but fixed; $\k$ sufficiently large, depending on $c$, but also fixed; $d$ even and going to infinity. This is essentially the setting of \cref{th:high}.

Naively, the guiding principle is ``look for percolation via unsaturated sites'', which is a non-monotone event and thus prohibits perturbative arguments. Intuitively, for our choice of parameters each site should have less than $\k$ edges with $U_e\le t$ (we call such edges feasible) with high probability (since $\k\gg 2c$ and the degree of each vertex is approximately Poisson with parameter $2c$). Discarding the remaining vertices (called saturated), we only need to show that edges are still open with fairly high probability. Fortunately, the information that a vertex was not saturated is not significant, as this event is likely, so the edges of those vertices should almost form an independent Bernoulli bond percolation. It is then not unreasonable to hope that the resulting nearly independent mixed site-bond percolation with site parameter close to $1$ and bond parameter close to $t$ would be supercritical, as it is known that the critical probability of bond percolation on $\bbL^d$ satisfies $\pc(d)=(1+o(1))/2d<c/d=t$ \cite{Kesten90}. Unfortunately, we could not formalise this intuition and rather take several detours, while keeping the same guideline.

The first technique we rely on originates from a classical work of Holley and Liggett \cite{Holley78}, where it was used to prove an upper bound of order $1/d$ on the critical parameter of the contact process in high dimensions. Similarly to \cite{Holley78} we map $\bbZ^d$ to $\bbZ^2$ with each two neighbours connected by $\lfloor d/2\rfloor$ edges in one direction and $\lfloor d/2\rfloor$ in the opposite one as follows. We split the $d$ vectors of the canonical basis of $\bbZ^d$ in two halves, viewing the first $\lfloor d/2\rfloor$ as pointing east (their opposites point west), while the other half point north (see \cref{eq:Fi}).

There are several advantages to working in two dimensions rather than directly on $\bbL^d$. Firstly, the control we have on mixed percolation deteriorates quickly with dimension. Secondly, exploring only few of the edges around a vertex allows us to keep the distributions of $U_e$ close to their original \emph{i.i.d.} uniform laws despite the dependencies. In particular, under this mapping the percolation model acquires the ``finite energy'' property, though we will not use it explicitly.

We build an exploration of a part of the cluster of $0$ in the original constrained percolation model, so as to compare it with mixed site-bond percolation on $\bbL^2$ via the mapping described above. The exploration should rather be viewed in $\bbL^2$ as we will never visit the same site there twice. Starting with $0$ as our only active site we repeat the following steps until we run out of untreated active sites. We first verify if the active site under consideration is saturated (has more than $\k$ feasible edges). If it is, we close it and move on. Notice that we may only have discovered $3$ feasible edges to that site previously, since it has degree $4$ in $\bbL^2$ and there is no point in considering vertices all of whose neighbours are already in the cluster of $0$ in $\bbL^2$. Thus, the vertex remains open with high probability, as $\k-3\gg 2c$.

Knowing that a vertex is open does not tell us much about whether or not we can reach its neighbours in $\bbL^2$ via feasible edges. We activate each of the inactive neighbouring vertices if we find at least one feasible edge among the $\lfloor d/2\rfloor$ from our current position. Since it suffices to find one feasible bond per neighbour, the next neighbour is not heavily penalised by the previous one becoming active, as $\k-3\gg 2c$. Thus, the probability of a neighbour being activated, given that our original site remained open, is close to the probability that a Poisson random variable with parameter $c/2$ (as there are $\lfloor d/2\rfloor$ edges) is non-zero, which is close to $1$ for $c$ sufficiently large.

Summing up, when viewed in $\bbL^2$, the exploration opens each active site with probability close to $1$ and then activates each of its inactive neighbours with probability close to $1$. This clearly corresponds to the exploration of the cluster of $0$ in $\bbL^2$ in a mixed site-bond percolation with both parameters close to $1$. Since this is easily seen to be supercritical, we obtain that with positive probability there is an infinite path in $\bbL^d$ whose edges are all feasible and whose sites are all unsaturated, which concludes the proof that $t\ge \tc^{\k}$.

For ``finite'' values of $c$, $\k$ and $d$ we aim for a comparison with a site-bond percolation with bond parameter slightly larger than $1/2$ and site parameter very close to $1$. We then use a refinement of a result of Chayes and Schonmann \cite{Chayes00} established in \cref{app} to prove that the parameters are indeed supercritical. In order to prove \cref{th:high} we employ the same strategy, with the difference that we now divide the $d$ directions into $d'\ll\min(d,\k)$ groups and thus reduce the problem to mixed site-bond percolation on $\bbL^{d'}$. We then use a simpler qualitative version of the result of \cite{Chayes00} as obtained already by Liggett, Schonmann and Stacey \cite{Liggett97} together with Kesten's result \cite{Kesten90} affirming that $d'\cdot \pc(d')\to1/2$ for ordinary bond percolation as $d'\to\infty$.

It is important to note that our argument is intrinsically non-monotone and it is therefore not possible to bring the matter down to a qualitative result on mixed site-bond percolation such as the classical theorem of Liggett, Stacey and Schonmann \cite{Liggett97}. Instead, we require a rather good quantitative bound on the critical curve of mixed percolation. This non-monotonicity is also the reason for obtaining quite strong non-perturbative upper bounds on $\tc^{\k}$ in \cref{th:general}, contrary to previous works \cites{DeLima20, Teodoro14}, but, on the downside, for the values of $\k$ and $d$ for which the resulting mixed percolation is subcritical for any choice of $t$, we recover no result at all.

\subsubsection{The cubic lattice}
\label{subsubsec:3d}
The proof of \cref{th:3d} will use some of the ingredients of \cref{th:general}. The two most important differences are that we will no longer systematically discard saturated vertices and that we will look for a comparison with two-dimensional bond percolation rather than mixed site-bond percolation. We will focus on $\bbL^3$, as $\bbL^2_{\boxtimes}$ is treated identically. We fix $\k=5$ and $t=0.62$ as in \cref{th:3d}.

This time no mapping is required to reduce $\bbL^3$ to $\bbL^2$, we rather directly look for percolation in the horizontal plane containing $0$. As it was pointed out in \cites{DeLima20,DoAmaral21}, the constrained percolation model does not have a clear monotonicity w.r.t.\ the underlying graph (see \cref{sec:open} below), even if the constraint $\k$ is adjusted accordingly, so we will not rely on any type of monotonicity. Instead, we will use the edges pointing out of the plane to save certain vertices which seem saturated in the plane.

We explore the cluster of $0$ in the plane, treating one active vertex $a$ at a time as follows. We activate the neighbours of $a$ if their bond to $a$ is feasible with the exception of the case in which $a$ is saturated (i.e.\ all 6 edges from $a$ are feasible). If $a$ is saturated, rather than closing it, we look at whether one of its edges going out of the plane happens to have $U_e$ larger than all the edges from $a$ in the plane. If that is the case, we activate all neighbours, while if it fails, we activate none.

Our goal is then to show that the net result of treating each vertex is activating each neighbour at least independently with probability $p>1/2$. Intuitively, the probability of activating all neighbours should be $t^k-t^{2+k}\frac{k}{k+2}$, where $k$ is the number of inactive neighbours. It is then reasonable to hope to be able to establish the desired stochastic domination with
\[p=\max_{k\in\{1,2,3\}}\left(t^k-t^{2+k}\frac{k}{k+2}\right)^{1/k}>1/2.\]

However, more care is needed, as we do not only look at the feasibility of edges, but also at the actual value of their $U_e$. This information may potentially accumulate, propagate and interfere with the probability that the unexplored edges out of the plane have larger $U_e$ than the ones already (partially) explored in the plane from the same vertex. Fortunately, carefully choosing what information to reveal, we are able to ensure that when we activate a vertex the corresponding $U_e$ is either uniformly distributed on $[0,t]$ (as we know it is feasible) or is further biased towards small values, which is in our favour when we compare it with edges out of the plane. This is quite natural, as the only information we may acquire on $U_e$ in addition to being feasible is that it is smaller than one of the edges out of the plane at one of its endpoints.

\section{General case---proof of \cref{th:general,th:high}}
\label{sec:high}
In this section we start by proving \cref{th:general}, assuming the results on mixed site-bond percolation from \cref{app}, namely \cref{cor:pc}, which will be used as a black box. We refer the reader to \cref{subsubsec:sketch:general} for a high-level sketch of the argument.
\begin{proof}[Proof of \cref{th:general}]
Let $t=c/d$ with $c=1.7$ and call an edge $e$ \emph{feasible} if $U_e\le t$. We consider the map
\begin{equation}
\label{eq:Fi}\F(x)=\left(\sum_{i=1}^{\lfloor d/2\rfloor}x_i,\sum_{i=d-\lfloor d/2\rfloor}^dx_i\right)
\end{equation}
from $\bbZ^d$ to $\bbZ^2$. We will build a supercritical mixed site-bond percolation process on $\bbL^2$ stochastically minorating the image of the cluster of $0$ in the constrained percolation model on $\bbL^d$ with parameters $\k$ and $t$. We will do so by exploring the cluster of $0$ in the following way.

We will construct $A_{n}\subset \bbZ^d$ the set of \emph{active} sites at time $n\in\bbN$ and the sets of \emph{open}, \emph{closed} and \emph{useless} vertices, $O_{n},C_{n},U_{n}\subset A_{n}$ respectively. For all $n\in\bbN$ we set $A'_{n}=\F(A_{n})$, $O'_{n}=\F(O_{n})$, $C'_{n}=\F(C_{n})$, and $U'_{n}=\F(U_{n})$. Unless otherwise stated, when incrementing $n$ all the above sets remain unchanged.

\begin{algo}
\label{algo:general}
Initialise $A_0=\{0\}$, $U_0=C_0=O_0=\varnothing$ and $n=0$.

\begin{enumerate}[label=Step \arabic*,ref=Step \arabic*]
\item\label{step:1}
If $A_n=O_n\cup C_n\cup U_n$, then END. Otherwise, fix $a\in A_n\setminus(O_n\cup C_n\cup U_n)$. If $\F(a)$ has no neighbour outside $A'_{n}$, set $U_{n+1}=U_n\cup\{a\}$, increment $n$ and repeat \ref{step:1}. Otherwise,
    \begin{itemize}
        \item if $a$ has at most $\k$ feasible edges, set $O_{n+1}=O_n\cup\{a\}$, increment $n$ and go to \ref{step:2};
        \item otherwise, set $C_{n+1}=C_n\cup\{a\}$, increment $n$ and repeat \ref{step:1}.
    \end{itemize}
    It is important to note that we do not explore the state (feasible or not) of the edges of the vertex $a$, but just ask whether there are more than $\k$ feasible ones or not.

\item\label{step:2}
Let $\{o\}=O_n\setminus O_{n-1}$. Let $V'=\{v'_i,i\in I\}$ be the set of neighbours of $\F(o)$ which are not in $A'_n$. For each $i\in I$ explore the $\lfloor d/2\rfloor$ edges from $o$ to $\F^{-1}(v_i')$ one by one until a feasible edge is discovered. If such an edge is found, let $v_i$ be its endpoint (other than $o$). Set $A_{n+1}=A_n\cup\{v_i,i\in I\}$, increment $n$, and go to \ref{step:1}.
\end{enumerate}
\end{algo}

Let us make a few observations about this algorithm. First, the map $\F$ is always injective on $A_n$, since vertices $v'\in V'$ considered for activation in \ref{step:2} are not in $A_{n}'$ and at most one preimage by $\F$ of each $v'$ is activated, corresponding to the first feasible edge discovered. Furthermore, it is clear that all vertices in $O_n$ have at most $\k$ feasible edges, so feasible edges between vertices in $O_n$ are open in the constrained percolation. Moreover, by induction all vertices in $O_n$ are connected to $0$ (since each new active vertex is connected to an open one). In particular, if the algorithm does not finish, then $0$ belongs to an open infinite cluster in the constrained percolation model. On the other hand, in $\bbL^2$ the algorithm only considers neighbours of $O'_n$, so it never terminates if and only if $0$ is in an infinite cluster $\bigcup_n O'_n$ in $\bbL^2$. 

We next analyse what information we have on the feasibility of different edges. Clearly, nothing is known about edges not incident with any active vertex. Let $n>0$ and $a\in A_n\setminus(O_n\cup C_n\cup U_n)$ be the vertex considered by \cref{algo:general} in \ref{step:1}. It is not hard to see that $a$ became active (in \ref{step:2}) at the time when we discovered the first feasible edge $e(a)$ connecting $a$ to an open vertex. Assume that $a$ does not become useless (which is purely deterministic at the time of consideration of $a$). Then $a$ has at most two edges from other active vertices (since $\F$ is injective on $A_n$) and we have no information on its remaining (at least $2d-3$) edges other than $e(a)$. If $a$ is declared open, in \ref{step:2} we additionally know that it has at most $\k$ feasible edges (including the one, two or three edges we previously had some information on).

Let $B_{m,p}$ denote the cumulative distribution function of the binomial law with parameters $m$ and $p$ (which is a step function continuous to the right).
\begin{claim}
The probability that a vertex considered in \ref{step:1} and not declared useless becomes open, conditionally on the information revealed by the algorithm until that moment, is least $s:=B_{2d-3,t}(\k-3)$.
\end{claim}
\begin{proof}
There are $j\le3$ explored edges to active vertices and nothing is known about the other $2d-j$ edges, so the conditional probability we seek is at least $B_{2d-j,t}(\k-j)\ge s$.
\end{proof}
\begin{claim}
The probability that a neighbour $v'$ of $\F(o)$ in \ref{step:2} becomes active, conditionally on the information revealed by the algorithm until the moment when $v'$ is considered, is at least \[b:=1-\frac{(1-t)^{\lfloor d/2\rfloor}}{s}\ge 1-\frac{1}{s\cdot\exp\left(\frac{c}{2}\left(1-\frac 1 d\right)\right)}.\]
\end{claim}
\begin{proof}
Let $i$ be the number of neighbours already activated by $o$. Observe that we have revealed $i+1$ feasible edges of $o$ ($e(o)$ used to make $o$ active and one for each neighbour activated by $o$ until now during \ref{step:2}) as well as several unfeasible edges. Additionally, we have information on $j\le 2$ more of its edges (to active vertices not activated by $o$). However, $i+j+1\le 3$, since there are only $4$ neighbours of $o'$ and $\F$ is injective on $A_n$.

Let us denote by $X_1,\dots, X_{k}$ with $\lfloor d/2\rfloor \le k\le 2d-3$ the $\1_{U_e\le t}$ for the unexplored edges $e$ of $o$, labelled so that $X_1,\dots,X_{\lfloor d/2\rfloor}$ correspond to edges from $o$ to $\F^{-1}(v')$. The $X_l$ are \emph{i.i.d.} Bernoulli variables with parameter $t$. We further define $X_{k+1},\dots,X_{2d}$ similarly for the remaining edges from $o$. Then $2d-k-i-j-1$ of the latter $X_l$ are already known to be $0$ and up to reordering, we assume them to be $X_{k+i+j+2},\dots,X_{2d}$. In total, the probability that $v'$ is not activated by $o$ is 
\begin{multline*}
\bbP\left(\sum_{l=1}^{\lfloor d/2\rfloor} X_l=0\left|\sum_{l=1}^{k+i+j+1}X_l\le\k\right.\right)\\
\le\max_{m\in[1,i+j+1]}\bbP\left(\sum_{l=1}^{\lfloor d/2\rfloor} X_l=0\left|\sum_{l=1}^{k}X_l\le\k-m\right.\right)\le\frac{B_{\lfloor d/2\rfloor,t}(0)}{B_{k,t}(\k-i-j-1)}.
\end{multline*}
Thus, it suffices to note that 
\[B_{k,t}(\k-i-j-1)\ge B_{2d-i-j-1,t}(\k-i-j-1)\ge B_{2d-3,t}(\k-3)=s.\qedhere\] 
\end{proof}

Observe that $s$ and $b$ are increasing in $\k$, so it suffices to treat $\k=10$. Let us note that if we only wanted to prove that $\tc^{\k}(d)\le c/d$ for $d$ large enough, we are already done by \cref{cor:pc} and the fact that \begin{align*}\lim_{d\to\infty}s&{}= P_{2c}(\k-3)\approx0.9770, &\lim_{d\to\infty}b&{}=1-\frac{e^{-c/2}}{P_{2c}(\k-3)}\approx 0.5625,\end{align*}
where $P_\l$ denotes the cumulative distribution function of a Poisson random variable with parameter $\l$. In order to obtain the desired result for all $d$, we will need a quantitative version of this convergence. 

We claim that for all $d>\k/2$ we have $s\ge 0.9765$ and $b\ge 0.5622$. Indeed, one may verify these inequalities directly for $d\le 4000$ (by computer) and, for $d>4000$ use the facts that 
\[s\ge B_{2d,t}(\k-3)\ge P_{2c}(7)-\frac{c\left(1-e^{-2c}\right)}{4000}\]
by Chen's inequality (see e.g.\ \cite{Steele94}*{Eq. (5.5)}) and so
\[b\ge 1-\frac{e^{-\frac{c}{2}\left(1-\frac{1}{4000}\right)}}{P_{2c}(7)-\frac{c\left(1-e^{-2c}\right)}{4000}}.\]

From the above it remains to check that mixed percolation with site and bond parameters $s\ge 0.9765$ and $b\ge0.5622$ respectively in two dimensions does percolate with positive probability, which follows directly from \cref{cor:pc}.

Turning to the specific low-dimensional cases in the statement of \cref{th:general}, the same proof applies with the corresponding sets of parameters. Indeed, in all cases we have either $s\ge 0.9809$ and $b\ge 0.5596$ or $s\ge0.9708$ and $b\ge0.5806$, which are supercritical by \cref{cor:pc}.
\end{proof}

We next explain the minor modifications needed in the proof above to establish \cref{th:high}.
\begin{proof}[Proof of \cref{th:high}]
As explained in \cref{subsec:results}, $\tc^{\k}(d)\ge 1/(2d-1)$, so it suffices to prove that for any $c>1/2$ and $\k$ and $d$ large enough depending on $c$ we have $\tc^{\k}(d)\le c/d$. Let us fix $c>1/2$, $d'$ large enough depending on $c$, so that $\pc(d')<1-e^{-c/d'}$ for ordinary bond percolation, which is possible, since $\lim_{d'}d'\pc(d')=1/2$ \cite{Kesten90}. We then fix $\k,d$ large enough depending on $c$ and $d'$ and set $t=c/d$.

Instead of \cref{eq:Fi}, we consider the map
\[\F:\bbZ^d\to\bbZ^{d'}:(x_i)_{i=1}^{d}\mapsto\left(\sum_{j=1}^{\lfloor d/d'\rfloor}x_{j+(i-1)\lfloor d/d'\rfloor}\right)_{i=1}^{d'}.\]
We then proceed exactly as in the proof of \cref{th:general} to establish a comparison with mixed site-bond percolation on $\bbL^{d'}$ with parameters \begin{align*}
s&{}=B_{2d-(2d'-1),t}(\k-(2d'-1)),\\
b&{}=1-\frac{(1-t)^{\lfloor d/d'\rfloor}}{s}.
\end{align*}
By the Poisson approximation, letting $d,\k\to\infty$ (regardless of the relationship between the two), while keeping $d'$ fixed, we have $s\to 1$ and $b\to 1-e^{-c/d'}>\pc(d')$. In particular, taking $d$ and $\k$ large enough we have $s\ge 1-\varepsilon$ and $b\ge \pc(d')+\d$ for any $\varepsilon,\d>0$ small enough depending only on $d'$.

Considering site percolation on $\bbL^{d'}$ with parameter $s\ge 1-\varepsilon$, by \cite{Liggett97} we have that for $\varepsilon$ small enough depending on $d'$ and $\d$ it stochastically dominates ordinary bond percolation with parameter $b'=1-\d$. We may then conclude that mixed site-bond percolation on $\bbL^{d'}$ with parameters $s$ and $b$ stochastically dominates pure bond percolation with parameter $bb'> \pc(d')$, which concludes the proof.
\end{proof}

\section{Low dimensional models---proof of \cref{th:3d}}
\label{sec:low}
In this section we prove \cref{th:3d}, refining our strategy from \cref{sec:high} as outlined in \cref{subsubsec:3d}.
\begin{proof}[Proof of \cref{th:3d}]
Let us begin by treating the cubic lattice, from which the two-dimensional result on $\bbL^2_{\boxtimes}$ will follow immediately.

For any vertex $v\in \bbZ^3$ we denote by $E_v=\{uv\in \bbE^3\}$ the set of edges from $v$. Denote by $P$ the plane $\bbZ^2\times\{0\}\subset\bbZ^3$. Our aim will be to establish a comparison with supercritical bond percolation in $P$. Let $\k=5$, $t=0.62$ and call an edge \emph{feasible} if $U_e<t$. We will explore the edges with at least one vertex in $P$ according to the following somewhat improved version of \cref{algo:general}. 

We will construct the sets of \emph{active, open} and \emph{closed} sites $A_n\subset P$, $O_n,C_n\subset A_n$ respectively, as well as the sets $B_n\subset \{uv\in\bbE^3:u\in P,v\in P\}$ and $S_n\subset \{uv\in \bbE^3:u\in P\}$ of \emph{boundary} and \emph{spoilt} edges respectively, as follows (see \cref{fig:algo:3d} for an example). Unless otherwise stated, when incrementing $n$ all the above sets remain unchanged. Whenever an edge $e$ becomes spoilt, we reveal the value of $U_e$.
\begin{figure}
    \centering
    \begin{tikzpicture}[x=2.0cm,y=2.0cm]
\draw [ultra thick] (0,0)-- (1,0);
\draw [ultra thick] (0,0)-- (0,1);
\draw [ultra thick] (1,0)-- (1,1);
\draw [ultra thick] (1,0)-- (1,-1);
\draw [ultra thick] (1,-1)-- (0,-1);
\draw [ultra thick] (1,-1)-- (2,-1);
\draw [ultra thick] (1,1)-- (2,1);
\draw [ultra thick] (2,1)-- (3,1);
\draw [ultra thick,color=gray] (3,1)-- (3,0);
\draw [ultra thick] (3,1)-- (4,1);
\draw [ultra thick] (4,1)-- (4,0);
\draw [ultra thick] (4,0)-- (4,-1);
\draw [dashed] (0.8,-1.2)-- (1,-1);
\draw (1,-1)-- (1.2,-0.8);
\draw (1.8,-1.2)-- (2,-1);
\draw (2,-1)-- (2.2,-0.8);
\draw [ultra thick] (1,0)-- (2,0);
\draw (1,0)-- (1.2,0.2);
\draw (1,0)-- (0.8,-0.2);
\draw (1,1)-- (0.8,0.8);
\draw (1,1)-- (1.2,1.2);
\draw [dashed] (0,0)-- (0.2,0.2);
\draw [dashed] (0,0)-- (-0.2,-0.2);
\draw [ultra thick,color=gray] (4,-1)-- (3,-1);
\draw [dashed] (2,1)-- (1.8,0.8);
\draw [dashed] (2,1)-- (2.2,1.2);
\draw [dashed] (3,1)-- (2.8,0.8);
\draw (3,1)-- (3.2,1.2);
\draw (4,1)-- (3.8,0.8);
\draw (4,1)-- (4.2,1.2);
\draw (4,0)-- (3.8,-0.2);
\draw (4,0)-- (4.2,0.2);
\draw [dashed] (4,-1)-- (3.8,-1.2);
\draw (4,-1)-- (4.2,-0.8);
\draw (0,1)-- (-0.2,0.8);
\draw (0,1)-- (0.2,1.2);
\draw [ultra thick,color=gray] (4,0)-- (5,0);
\draw [dashed] (0,0)-- (-1,0);
\draw [dashed] (2,0)-- (3,0);
\draw [dashed] (2,1)-- (2,2);
\draw [dashed] (1,1)-- (1,2);
\draw [ultra thick,color=gray] (3,1)-- (3,2);
\draw [dashed] (4,1)-- (4,2);
\draw [dashed] (4,1)-- (5,1);
\draw [dashed] (4,-1)-- (5,-1);
\draw [dashed] (4,-1)-- (4,-2);
\draw [dashed] (1,-1)-- (1,-2);
\draw [dashed] (-1,-1)-- (0,-1);
\draw [dashed] (0,-1)-- (0,-2);
\draw [dashed] (0,-1)-- (0.2,-0.8);
\draw (0,-1)-- (-0.2,-1.2);
\draw (2,0)-- (1.8,-0.2);
\draw [dashed] (2,0)-- (2.2,0.2);
\draw (2,-1)-- (2,-2);
\draw (2,-1)-- (3,-1);
\draw (2,-1)-- (2,0);
\draw (0,1)-- (0,2);
\draw (0,1)-- (1,1);
\draw (0,1)-- (-1,1);
\draw (3,0)-- (4,0);
\draw [dashed] (0,-1)-- (0,0);
\draw [dashed] (2,1)-- (2,0);
\draw [rotate around={45:(1.1,0.1)},ultra thick] (1.1,0.1) ellipse (0.32cm and 0.16cm);
\draw [rotate around={45:(3.9,-0.1)},ultra thick] (3.9,-0.1) ellipse (0.32cm and 0.16cm);

\fill (3,-1) circle (3pt);
\fill (5,0) circle (3pt);
\fill (3,0) circle (3pt);
\fill (3,2) circle (3pt);
\draw (0,1) node[cross,ultra thick] {};
\draw (2,-1) node[cross,ultra thick] {};
\draw (3,-0.4) node[label=$a$] {};
\draw (2.8,0.25) node[label=$b(a)$] {};
\draw (0.9,-1) node[label=$0$] {};
\draw (0.5,-0.4) node[label=$e_1$] {};
\draw (0.9,0.25) node[label=$e_2$] {};
\draw (1.5,-0.4) node[label=$e_3$] {};
\draw (1.25,0.15) node[label=$e$] {};
\draw (0.8,-0.75) node[label=$b(v)$] {};
\draw (1.1,-0.3) node[label=$v$] {};
\draw (4.5,-0.4) node[label=$e'_1$,anchor=north east] {};
\draw (3.9,-0.75) node[label=$e'_2$,anchor=north east] {};
\draw (3.7,-0.4) node[label=$e'$,anchor=north east] {};
\draw (3.8,0.25) node[label=$b(v')$] {};
\draw (4.1,-0.3) node[label=$v'$] {};
\end{tikzpicture}
    \caption{Illustration of \cref{algo:3d}. The currently discovered part of the cluster of the origin is thickened. The active sites which are neither open or closed yet, $A_n\setminus (O_n\cup C_n)$ are represented by dots, the closed ones, $C_n$, are crossed out, the open sites, $O_n$, are all the remaining vertices of the thick cluster, the boundary edges, $B_n$, are drawn in grey, the spoilt ones, $S_n$, are black. The solid lines represent feasible edges, while dashed ones are not feasible. Notice that the vertex $a$ will surely become closed when it is examined, as it has no inactive neighbours in $P$. The two circled edges going out of the plane $P$ were used to save their respective vertices from being closed due to having all their 6 edges feasible. Namely, we have $U_{e}\ge\max(U_{e_1},U_{e_2},U_{e_3},U_{b(v)})$ and $U_{e'}\ge\max(U_{e'_1},U_{e'_2},U_{b(v')})$.
    \label{fig:algo:3d}}
\end{figure}
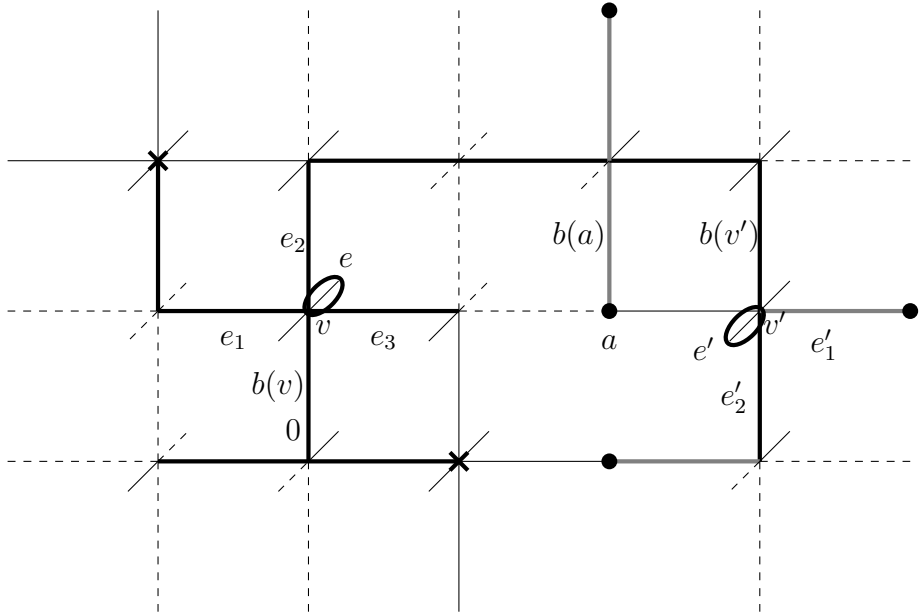
\begin{algo}
\label{algo:3d}
Initialise $A_0=\{0\}$, $O_0=C_0=B_0=S_n=\varnothing$ and $n=0$. REPEAT the following.
If $A_n=O_n\cup C_n$, then END. Otherwise, fix $a\in A_n\setminus(O_n\cup C_n)$ and let $b(a)$ denote the edge in $B_n$ with endpoint $a$ (there will always be exactly one such edge except for $a=0$, in which case we make the convention $\{b(0)\}=\varnothing$). If $a$ has no neighbour in $P\setminus A_n$, set $C_{n+1}=C_n\cup\{a\}$, $B_{n+1}=B_n\setminus\{b(a)\}$, $S_{n+1}=S_n\cup E_a$, increment $n$ and go back to REPEAT. Otherwise, for each edge $av$ in $E_a\setminus S_n$ we reveal whether it is feasible or not, let $\G(a)$ denote the set of vertices $v\in P\setminus A_n$ such that $av$ is feasible, set $\Eo_a=\{av,v\in\G(a)\}$ and proceed as follows.
\begin{itemize}
    \item If at most $5$ edges in $E_a$ are feasible, set  $A_{n+1}=A_n\cup\G(a)$, $O_{n+1}=O_n\cup\{a\}$, $B_{n+1}=(B_n\setminus\{b(a)\})\cup\Eo_a$, $S_{n+1}=S_n\cup \left(E_a\setminus \Eo_a\right)$, increment $n$ and go back to REPEAT.
    \item Otherwise, let $au$ be the edge with largest $U_{au}$ among \[\Eo_a\cup\{b(a)\}\cup\{av, v\not\in P\}.\]
    If $u\in P$, then set $C_{n+1}=C_n\cup\{a\}$, $B_{n+1}=B_n\setminus \{b(a)\}$, $S_{n+1}=S_n\cup E_a$, increment $n$ and go back to REPEAT. Otherwise, set $A_{n+1}=A_n\cup\G(a)$, $O_{n+1}=O_n\cup\{a\}$, $B_{n+1}=(B_n\setminus\{b(a)\})\cup \Eo_a$, $S_{n+1}=S_n\cup (E_a\setminus\Eo_a)$, increment $n$ and go back to REPEAT.
\end{itemize}
\end{algo}

Let us make a few observations about this algorithm. First,  since any vertex $v$ is activated at most once, it is clear that $b(v)$ is well defined. Moreover, at the time of activation of a vertex $v$ the other endpoint $a$ of $b(v)$ becomes open, which guarantees that the constraint at $a$ is not violated by $b(v)$ (either there are not 6 feasible edges at $a$ or at least one of them is to be added after $b(v)$) and that edge is feasible. Thus, whenever a vertex $v$ becomes open, the edge $b(v)$ is known to be present in the constrained percolation, since we have also checked that the constraint at $v$ is not violated by that edge. Hence, all open vertices belong to the cluster of $0$. In particular, if the algorithm does not finish, then $0$ belongs to an infinite cluster. On the other hand, the algorithm only activates neighbours of open vertices, so it never terminates if and only if $0$ is in an infinite cluster $\bigcup_n O_n$ in $P$.

Further note that at any given time edges (with at least one end in $P$, as others will never be used) are divided in three categories: spoilt, boundary and \emph{unexplored}. We know nothing about the value $U_e$ for unexplored $e$, we know the exact value for spoilt $e$ and we will next assess boundary edges and show that we may view them as unexplored. Observe also that all non-boundary edges incident with open or closed vertices are spoilt, while all edges incident with a vertex $a\in A_n\setminus(O_n\cup C_n)$ except its boundary edge, $b(a)$ are unexplored.

\begin{lem}
\label{lem:decoupling}
For $n\ge 0$ conditionally on the information revealed by the algorithm until time $n$, the corresponding $\s$-algebra being denoted by $\cF_n$ (the random set $B_n$ is measurable w.r.t.\ $\cF_n$), the $(U_e)_{e\in B_n}$ are independent and each $U_e$ has a uniform law on $[0,p(e)]$ with $p(e)\le t$ measurable w.r.t.\ $\cF_n$.
\end{lem}
\begin{proof}
Observe that the connected components of $B_n$ are stars centered at open vertices. We will prove the statement by induction on $n$, so we assume it holds for a given $n$. 

At step $n$ \cref{algo:3d} reveals information only about the edges adjacent to a certain vertex $a\in A_n\setminus (O_n\cup C_n)$ and does not take into account any other edges. In particular, the values of $(U_{av})_{v\in P\setminus A_n}$, which were previously unexplored, are independent of $(U_e)_{e\in B_n\setminus\{b(a)\}}$ (conditionally on $\cF_n$) by induction hypothesis. If $a\in C_{n+1}$ there is nothing left to prove, since we have simply spoiled $E_a\setminus S_n$ (namely, $\cF_{n+1}=\s(\cF_n,(U_e)_{e\in E_a\setminus S_n})$) and these edges were either $b(a)$ or unexplored, so they were all independent of $(U_e)_{e\in B_n\setminus\{b(a)\}}$ by induction hypothesis. We next assume that $a\in O_{n+1}$ and consider two cases.

Assume first that there are at most $5$ feasible edges in $E_a$. Then we only explored which of the edges in $E_a\setminus S_n$, spoiled $E_a\setminus(\Eo_a\cup S_n)$ and made $\Eo_a$ boundary edges. In particular, we have \[\cF_{n+1}=\s\left(\cF_n,\Eo_a,(U_e)_{e\in E_a\setminus\left(\Eo_a\cup S_n\right)}\right)\]
(recall that $\Eo_a$ is a random set). Hence, conditionally on $\cF_{n+1}$, we only know that $U_e\le t$ for $e\in\Eo_a$ by definition of $\Eo_a$ and we are done.

Finally, assume that all six edges in $E_a$ are feasible, but the vertex $u$ from \cref{algo:3d} is not in $P$. In this case $\Eo_a=\{av,v\in P\setminus A_n\}$ and, conditionally on $\cF_{n+1}$ we only know that $U_e\le \max_{v\not\in P}(U_{av})$ for all $e\in \Eo_a$. Yet, $\max_{v\not\in P}(U_{av})$ is measurable w.r.t.\ $\cF_{n+1}$, since both such edges $av$ are in $S_{n+1}$, as $a\in O_{n+1}$. Finally, since all edges in $E_a$ are feasible by hypothesis, we obtain that $\max_{v\not\in P}(U_{av})\le t$ and we are done.
\end{proof}

We next establish the desired comparison with ordinary percolation.
\begin{lem}
\label{lem:comparison}
Fix $n>0$ and let $a\in A_n\setminus(O_n\cup C_n)$ be the vertex considered at that step. Let $X=\{v\in P\setminus A_n,va\in E_a\}$ be the set of sites which may be added to $A_n$ at this step. Conditionally on $\cF_n$ the variables $(\1_{x\in A_{n+1}\setminus A_n})_{x\in X}$ stochastically dominate \emph{i.i.d.} Bernoulli variables with parameter $p>0.5$.
\end{lem}
\begin{proof}
If $X$ is empty there is nothing to prove, so we have $|X|\in\{1,2,3\}$. Let $k$ denote the number of edges $av\in S_n$. Without loss of generality we will assume that $k=3-|X|$, as otherwise we may simply condition on the value of $U_{av}$ for $v\in A_n\setminus (O_n\cup C_n)$.

If any of the $U_{av}>t$ for $av\in S_n$, then $B_{n+1}\setminus B_n$ is simply the set of feasible edges from $a$ to $X$, which gives that $(\1_{x\in A_{n+1}\setminus A_n})_{x\in X}$ are exactly \emph{i.i.d.} Bernoulli with parameter $t$, since these edges are unexplored.

Let us assume that all $3-|X|$ edges $av\in S_n$ are feasible. Let $N=\sum_{x\in X}\1_{x\in A_{n+1}\setminus A_n}$. By symmetry it suffices to show that $N$ stochastically dominates a binomial random variable with parameters $|X|$ and $p$. 

Let us fix $|X|=3$ for a start. We claim that
\begin{align}
\label{eq:N1}\bbP(N=1|\cF_n)&{}=3t(1-t)^{2},\\
\label{eq:N2}\bbP(N=2|\cF_n)&{}=3t^2(1-t),\\
\label{eq:N3}\bbP(N=3|\cF_n)&{}\ge t^3-t^5+\frac{2}{6}t^5.
\end{align}
\cref{eq:N1,eq:N2} follow directly from \cref{algo:3d}, since $N<3$ guarantees that there are at most 5 feasible edges at $a$. To check \cref{eq:N3}, we notice that we need all three edges $ax$ for $x\in X$ to be feasible; in the case when all edges at $a$ are feasible (we already know from $\cF_n$ that $b(a)$ is, but the other $2+|X|$ edges not in $S_n$ are unexplored) we still have a chance that the largest $U_{av}$ among $b(a)$ and the $2+|X|$ unexplored edges is achieved for $v\not\in P$. Using \cref{lem:decoupling} the latter probability is indeed at least $2/(3+|X|)$. It then suffices to check that 
\begin{align*}
3t(1-t)^{2}+3t^2(1-t)+ t^3-t^5+\frac{2}{6}t^5&{}> 1-(1-p)^3,\\
3t^2(1-t)+t^3-t^5+\frac{2}{6}t^5&{}> p^2(p+3(1-p)),\\
t^3-t^5+\frac{2}{6}t^5&{}> p^3,
\end{align*}
for $p=0.5$, which is immediate.

For $|X|\in\{1,2\}$ we proceed identically, reaching the inequalities
\begin{align*}
2t(1-t)+t^2-t^4+\frac{2}{5}t^4&{}> p^2+2p(1-p),\\
t^2-t^4+\frac{2}{5}t^4&{}> p^2,\\
t-t^3+\frac{2}{4}t^3&{}> p,
\end{align*}
which are again easily verified for $p=0.5$.
\end{proof}
With \cref{lem:comparison} we are ready to conclude the proof of \cref{th:3d} for $\bbL^3$. Indeed, it follows that one can couple the exploration of \cref{algo:3d} with an exploration of the cluster of $0$ in bond percolation in $P$ with parameter $p>0.5$ in such a way the set of eventually active sites contains the cluster of $0$ in the latter percolation model. Since the critical probability of bond percolation in two dimensions is $1/2$ \cite{Kesten80}, this concludes our proof.

In order to deal with the square lattice with diagonals added, $\bbL^2_{\boxtimes}$, it suffices to consider the edges $(x,x+(1,1))$ and $(x,x+(1,-1))$ as analogues of $(x,x+(0,0,1))$ and $(x,x+(0,0,-1))$ in the cubic lattice.
\end{proof}

\begin{rem}
Applying an analogous argument to the triangular or checkerboard lattices with $\k=5$ does not quite work as it stands, since we only manage to compare the constrained degree percolation model (with $t=0.678$) with bond percolation on $\bbL^2$ with parameter $p=0.47$, which is subcritical. However, it is likely that working slightly more, one could give a nontrivial result in that setting as well.
\end{rem}

\section{Open problems}
\label{sec:open}
Several questions can arise concerning this model. We can consider other graphs or allow the constraint $\kappa (v)$ to be a function of the vertex set, for example, but in the context of the present work, we would like to mention some open problems concerning the critical time $\tc$ for the hypercubic lattice. Some of these questions were already stated in \cites{DoAmaral21,DeLima20}. 

In ordinary Bernoulli percolation on $\bbL^d$, there is a trivial coupling that shows that the percolation threshold is a non-increasing function of the dimension $d$. This same coupling does not work to show that the critical time is a non-increasing function of $d$, which seems to be true.
\begin{ques}
For $\kappa$ fixed, is the function $\tc^{\k}(d)$ non-increasing in $d$?
\end{ques}

Still concerning the monotonicity of $\tc$, keeping the dimension fixed, one may ask whether $\tc^{\k}(d)$ is monotone in $\k$. For example, it was shown in \cite{DeLima20}, it holds that $t_{\mathrm{c}}^2(2)=+\infty$, $\frac{1}{2}<\tc^3(2)<1$ and it is known from \cite{Kesten80} that $\tc^4(2)=\frac{1}{2}$. Numerical support for the following conjecture was provided in \cite{DoAmaral21} for $d\in\{3,4\}$.

\begin{conj}
For all $d\ge 3$ the function $\tc^{\k}(d)$ is non-increasing in $\k$.
\end{conj}
If the answer to the previous question is affirmative, it is logical to define the \emph{critical constraint} $\kappa_{\mathrm{c}}(d):=\min\{\kappa: \tc^\kappa (d)<1\}$.In this language, \cref{th:general,th:3d} provide that $\k_{\mathrm{c}}(d)\le \min(10,2d-1)$ for all $d\ge 3$. In \cite{DoAmaral21} some simulations were performed for dimensions $d=3$ and $4$ that support the following conjecture.
\begin{conj}
For all $d\ge 3$ it hold that $\kappa_{\mathrm{c}}(d)=3$. 
\end{conj}
We remark that in \cite{DeLima20} an analogous result was proved for the regular trees.

Turning to high dimensions, in view of our treatment in \cref{th:general,th:high}, it seems reasonable to expect a positive answer to the following question.
\begin{ques}
Does $\lim_{d\to\infty}d\cdot \tc^{\k}(d)$ exist for all $\k\ge 3$?
\end{ques}

\section*{Acknowledgements}
I.H. was supported by ERC Starting Grant 680275 MALIG. B.N.B.L. was supported in part by CNPq grant 305811/2018-5 and FAPERJ (Pronex E-26/010.001269/2016). The authors would like to thank the organisers of the Bernoulli-IMS One World Symposium 2020, which was the occasion for them to ``meet'' and introduce the first author to the model. Thanks are also due to La\"etitia Comminges et Djalil Chafa\"i for pointing us to \cite{Steele94} and to Lyuben Lichev for proofreading.

\appendix
\section{Mixed site-bond percolation on $\bbL^2$}
\label{app}
The lemma below for a mixed Bernoulli percolation on the square lattice, $\mathbb{L}^2$, where sites and bonds are open independently with probabilities $s$ and $b$, respectively, is essentially Proposition 2.1 of \cite{Chayes00}. We make a slight modification in the proof that allows us to improve the Lipschitz constant.

Let us define $P_{s,b}$ as the probability measure for this site-bond percolation model and $\theta_n(s,b):=P_{s,b}(0\leftrightarrow\partial B_{n+1})$, where $B_n=\{x\in\bbZ^2 :\|x\|_1=n\}$ and $(0\leftrightarrow A)$ is the set of configurations $\omega\in\{0,1\}^{\bbZ^2\cup\bbE^2}$ such that there is a path $\gamma=(x_0,x_1,\dots,x_{k+1})$ where $x_0=0$, $x_{k+1}\in A$, $x_i$ is open for all $i\in\{1,\dots, k\}$ and $x_ix_{i+1}$ is open for all $i\in\{0,\dots, k\}$.

\begin{lem}\label{lem:lipschitz} For the site-bond Bernoulli percolation model on the square lattice $\bbL^2$, it holds that
\[\frac{\partial\theta_n(s,b)}{\partial s}\leq \frac{4-3b}{2s(1-b)}\cdot\frac{\partial\theta_n(s,b)}{\partial b}.\]
\end{lem}
\begin{proof}
Let $A_n$ be the event $(0\leftrightarrow\partial B_{n+1})$. Given $x\in\bbZ^2$ and $e\in\bbE^2$, let $\delta_xA_n$ and $\delta_e A_n$ be the events where $x$ and $e$ are pivotal for the event $A_n$, respectively, and $A_{x,n}:=\d_xA_n\cap \{x\mbox{ is open}\}$.

By Russo's formula, we have that
\begin{equation}
    \frac{\partial\theta_n(s,b)}{\partial s}=\sum_{x\in B_n}P_{s,b}(\delta_x A_n)\quad \textrm{and} \quad \frac{\partial\theta_n(s,b)}{\partial b}=\sum_{e\in \bbE(B_{n+1})}P_{s,b}(\delta_e A_n).
    \end{equation}
    Define $E_x=\{e\in \bbE(B_{n+1}):x\in e\}$. Thus, observing that each bond in $\bbE(B_{n+1})$ contains at most two vertices in $B_n$, it is enough to prove that:
    \begin{equation}\label{sitebond}
     P_{s,b}(\delta_x A_n)\leq \frac{4-3b}{4s(1-b)}\sum_{e\in E_x}P_{s,b}(\delta_e A_n) .
    \end{equation}
Given $\omega\in A_{x,n}$, define $\Ei_{x,\omega}$ as the set of bonds in $E_x$ such that, in the configuration $\omega$, if $y$ is the other end-vertex distinct of $x$, $y$ is either the origin or else $y$ is open and connected to the origin by an open path that does not pass through $x$. Analogously, we define $\Eo_{x,\omega}$ as the set of bonds in $E_x$ such that, in the configuration $\omega$, if $y$ is the other end-vertex distinct of $x$, $y$ belongs to $\partial B_{n+1}$ or else $y$ is open and connected to $\partial B_{n+1}$ by an open path that does not pass through $x$. Observe that $\Ei_{x,\omega}\neq\varnothing$ and $\Eo_{x,\omega}\neq\varnothing$, we define $A_{x,n}^{i,j}:=\{\omega\in A_{x,n}:|\Ei_{x,\omega}|=i,|\Eo_{x,\omega}|=j\}$. The possible index set is $I=\{(1,1),(1,2),(1,3),(2,1),(3,1),(2,2)\}$, so that $A_{x,n}=\bigcup_{(i,j)\in I} A_{x,n}^{i,j}$ and $A_{x,n}\setminus A_{x,n}^{2,2}\subset\bigcup_{e\in E_x} \delta_e A_n.$

We further define 
\begin{align*}
A_{x,n}^{1,2,|}&{}=\left\{\o\in A_{x,n}^{1,2},\text{ $\Ei$ consists of a vertical bond}\right\}
\\A_{x,n}^{1,2,-}&{}=\left\{\o\in A_{x,n}^{1,2},\text{ $\Ei$ consists of a horizontal bond}\right\}
\end{align*}
and similarly for $A_{x,n}^{2,1,|}$ and $A_{x,n}^{2,1,-}$, replacing $\mathrm{in}$ by $\mathrm{out}$.
By planarity it is impossible to have $\Ei_{x,\o}$ consisting of two vertical bonds and $\Eo_{x,\o}$ consisting of two horizontal ones or vice versa. Thus, given a configuration in $A_{x,n}^{2,2}$, by closing each of the bonds in $E_x$, we obtain a configuration in each of the four events, $A_{x,n}^{1,2,|}$, etc., above, which are manifestly disjoint. Therefore,
\begin{align*}4P_{s,b}\left(A_{x,n}^{2,2}\right)&{}\le\frac{b}{1-b}P_{s,b}\left(A_{x,n}^{1,2,|}\sqcup A_{x,n}^{1,2,-}\sqcup A_{x,n}^{2,1,|}\sqcup A_{x,n}^{2,1,-}\right)\\
&{}\le \frac{b}{1-b}P_{s,b}\left(A_{x,n}\setminus A^{2,2}_{x,n}\right).
\end{align*}
Observing that $A_{x,n}^{i,j}\subset \bigcup_{e\in E_x}\delta_e A_n$ for all $(i,j)\neq (2,2)$, we get
\begin{align*}
    P_{s,b}(\delta_x A_n)= {}& s^{-1}.P_{s,b}(A_{x,n})=s^{-1}.\left(P_{s,b}\left(A_{x,n}\setminus A_{x,n}^{2,2}\right)+P_{s,b}\left(A_{x,n}^{2,2}\right)\right)\\
    \le{}& s^{-1}.P_{s,b}\left(A_{x,n}\setminus A_{x,n}^{2,2}\right)\left(1+\frac{b}{4(1-b)}\right)\\
    \le{}& \frac{4-3b}{4s(1-b)}P_{s,b}\left(\bigcup_{e\in E_x} \delta_e A_n\right).
    \end{align*}
This proves \cref{sitebond} and concludes the proof.
\end{proof}

\begin{cor}
\label{cor:pc}
For the mixed site-bond percolation on $\bbL^2$, it holds that
\begin{equation}
\label{eq:cor:pc}\sc(b)\le \exp\left(-\frac{2}{3}\left(b-\frac{1}{2}+\frac{1}{3}\log\frac{8-6b}{5}\right)\right),
\end{equation}
where $\sc(b)=\inf\{s\in[0,1]:\lim_{n}\theta_n(s,b)>0\}$ (see \cref{fig:mixed}).
\end{cor}
\begin{proof}
Observing that the limit in $n$ of the gradient vector $\nabla\theta_n$ is orthogonal to the critical curve $\sc(b)$, by \cref{lem:lipschitz}, we have that the curve $\sc(b)$ is bounded from above by the solution of the differential equation
\[\frac{\mathrm{d}s}{\mathrm{d}b}=-\frac{2s(1-b)}{4-3b}\] with $s(\frac{1}{2})=1$.
\end{proof}

\begin{figure}
\centering
\begin{tikzpicture}[line cap=round,line join=round,x=10.0cm,y=10.0cm]
\draw[->] (0.5,0.5) -- (1.05,0.5);
\foreach \x in {0.5,0.6,0.7,0.8,0.9,1}
\draw[shift={(\x,0.5)}] (0pt,2pt) -- (0pt,-2pt) node[below] {\footnotesize $\x$};
\draw[->] (0.5,0.5) -- (0.5,1.05);
\foreach \y in {0.5,0.6,0.7,0.8,0.9,1}
\draw[shift={(0.5,\y)}] (2pt,0pt) -- (-2pt,0pt) node[left] {\footnotesize $\y$};

\draw (1.05,0.5) node[below] {$b$};
\draw (0.5,1.05) node[left] {$s$};
\draw (0.5,1.05) node[right] {\cite{Kesten80}};
\draw (1,0.6795) node[right] {\cite{Wierman95}};
\draw (0.65,0.82) node[left] {\cite{Hovi96}};
\draw (1,0.5927) node[right] {\cite{Ziff92}};
\draw (0.9,0.755) node[right] {\cite{Hammersley80}};
\draw (0.9,0.89) node[above] {\cref{eq:cor:pc}};

\draw[ultra thick,dashed, smooth,samples=100,domain=0.6795:1] plot(\x,{0.6795/(\x)});
\draw[ultra thick, smooth,samples=100,domain=0.5:1] plot(\x,{2.718281828^(-(2/3)*((\x)-0.5+ln((8-6*(\x))/5)/3))});
\draw (1,1)-- (0.5,1);
\draw (1,1)-- (1,0.5);
\draw [very thin, dashed, color=gray] (0.74,0.5)-- (0.74,0.92);
\draw [thick, dotted] (0.5,1)-- (0.5336,0.95);
\draw [thick, dotted] (0.5336,0.95)-- (0.6158,0.85);
\draw [thick, dotted] (0.6158,0.85)-- (0.65,0.8156);
\draw [thick, dotted] (0.65,0.8156)-- (0.7262,0.75);
\draw [thick, dotted] (0.7262,0.75)-- (0.75,0.7321);
\draw [thick, dotted] (0.75,0.7321)-- (0.85,0.6673);
\draw [thick, dotted] (0.85,0.6673)-- (0.95,0.6152);
\draw [thick, dotted] (0.95,0.6152)-- (1,0.5927);
\end{tikzpicture}
\caption{\label{fig:mixed}Illustration of \cref{cor:pc,rem:pc}.}
\end{figure}

\begin{rem}
\label{rem:pc}
In \cref{fig:mixed} we represent the region of supercritical parameters $(s,b)$ (such that $\lim_n\theta_n(s,b)>0$). There result of \cref{cor:pc} is that all points above the thick solid line are supercritical. This should be compared and combined with a previous result by Hammersley \cite{Hammersley80} stating that $(s,b)$ is supercritical, whenever $sb\ge \sc(1)$. The latter quantity is the critical probability of site percolation on $\bbL^2$ and the best known upper bound on it to the authors' knowledge is $\sc(1)\le 0.6795$ due to Wierman \cite{Wierman95}. Combining these two results, one obtains that the region delimited by the dashed thick hyperbola is also supercritical. The crossover between the our \cref{cor:pc} of the Chayes--Schonmann approach \cite{Chayes00}, extrapolating from $\sc(0.5)=1$ \cite{Kesten80}, and the Hammersley--Wierman bound is for $b\approx 0.74$. Hence, our result offers an improvement for all $b\in(0.5,0.74)$, which is also the region of interest for us in \cref{sec:high}. For reference, the dotted broken line represents the result of (nonrigorous) numerical estimation of the actual critical curve $\sc$ \cite{Hovi96}*{Table I} and $\sc(1)$ \cite{Ziff92}.
\end{rem}



\end{document}